\documentclass[a4paper,11pt]{amsart}
\usepackage{hyperref,fancyhdr,mathrsfs,amsmath,amscd,amsthm,amsfonts,latexsym,amssymb,stmaryrd}
\usepackage[all]{xy}

\DeclareMathOperator{\rank}{rank}

\DeclareMathOperator{\dif}{d}

\newcommand{\ol}{\mathcal{O}}

\def \e{\eta}

\def \g{\gamma}

\def \G{\Gamma}

\def \phi{\varphi}
\def \Phi{\varPhi}
\def \p{\pi}
\def \r{\rho}

\def \C{\mathbb{C}\,}

\def\widecheckg{g^{\hspace*{-2.5pt}\vbox to 5pt{\hbox to
0pt{\LARGE$\check{}$}}}\hspace*{2pt}}

\def\widecheckl{\lambda^{\hspace*{-3.5pt}\vbox to 8pt{\hbox to
0pt{\LARGE$\check{}$}}}\hspace*{2pt}}

\begin{document}

\title{On the embeddings of the Riemann sphere\\ 
with nonnegative normal bundles} 
\author{Radu Pantilie} 
\thanks{This work is supported by a grant of the Ministery of Research and Innovation
CNCS-UEFISCDI, project no. PN-III-P4-ID-PCE-2016-0019, within PNCDI III} 
\email{\href{mailto:radu.pantilie@imar.ro}{radu.pantilie@imar.ro}} 
\address{R.~Pantilie, Institutul de Matematic\u a ``Simion~Stoilow'' al Academiei Rom\^ane,
C.P. 1-764, 014700, Bucure\c sti, Rom\^ania}
\subjclass[2010]{Primary 53C28, Secondary 53C26} 
\keywords{quaternionic geometry; twistor theory; embeddings of the Riemann sphere with nonnegative normal bundles}

\newtheorem{thm}{Theorem}[section]
\newtheorem{lem}[thm]{Lemma}
\newtheorem{cor}[thm]{Corollary}
\newtheorem{prop}[thm]{Proposition} 

\theoremstyle{definition}

\newtheorem{defn}[thm]{Definition}
\newtheorem{rem}[thm]{Remark}
\newtheorem{exm}[thm]{Example}

\numberwithin{equation}{section}
 
\begin{abstract} 
We describe the (complex) quaternionic geometry encoded by the embeddings of the Riemann sphere, 
with nonnegative normal bundles. 
\end{abstract} 

\maketitle 
\thispagestyle{empty}

\section*{Introduction} 

\indent 
There are problems that our generation can not hope to see them solved. Here are two such problems:\\ 
\indent 
\quad(1) classify the germs of the (known) geometric structures,\\ 
\indent 
\quad(2) classify the germs of the embeddings of compact complex manifolds.\\ 
\indent 
\emph{Twistor theory} provides a bridge between significant subclasses of the classes involved in (1) and (2)\,, above. 
From this perspective, in \emph{quaternionic geometry} the point is replaced by the Riemann sphere; that is, (up to a complexification) 
the objects are parameter spaces for locally complete families \cite{Kod} of embedded Riemann spheres. But, so far (see \cite{fq_2}\,, \cite{Pan-twistor_(co-)cr_q}\,, 
and the references therein), only for the families of Riemann spheres whose normal bundles split and are positive, there exists a fair understanding of the 
differential geometry of the corresponding parameter spaces (a positive vector bundle over the Riemann sphere is a complex analytic vector bundle 
whose Birkhoff--Grothendieck decomposition contains only terms of positive Chern numbers).\\ 
\indent 
In this note, we complete this picture by describing the geometric structures corresponding to the germs of embeddings of the Riemann sphere  
with nonnegative normal bundles. For this, we are led to consider \emph{principal $\r$-connections}, a straight generalization of the classical notion 
of principal connection, which appear (see Section \ref{section:gen_connection}, below) by shifting the emphasis from the tangent bundle to  
another vector bundle over the manifold (then $\r$ is a morphism from that bundle to the tangent bundle). 
The fact that the geometric structures may require another bundle, besides the tangent one, 
is not new (see \cite{Gua-AnnMath}\,). However, instead of a bracket, the quaternionic geometry requires a 
connection, to control the integrability \cite{Sal-dg_qm}\,. Our main result (Theorem \ref{thm:r-q}\,) shows that to integrate 
all the families of embedded Riemann spheres with nonnegative normal bundles we have to pass, from the classical connections, to the $\r$-connections, 
and from the (co-CR) quaternionic structures, to the \emph{$\r$-quaternionic structures} (see Section~\ref{section:r-q}\,).  
Furthermore, a bracket can be associated in this setting as well and, if $\r$ is surjective, the Jacobi identity is satisfied only if the corresponding 
$\r$-quaternionic manifold is obtained as a twistorial quotient of a (classical) quaternionic manifold. For example, the quaternionic geometry of the 
space of Veronese curves in the complex projective space is induced, through a twistorial submersion, from (the complexification of) a 
quaternionic manifold.\\ 
\indent 
As already mentioned, it was twistor theory who emphasized the richness of the differential geometry encoded by the embeddings of the Riemann sphere. 
Chronologically, there were the, by now, classical, anti-self-dual manifolds (see \cite{AtHiSi}\,), the three-dimensional Einstein--Weyl spaces 
\cite{Hit-complexmfds}\,, 
and the quaternionic manifolds \cite{Sal-dg_qm} (see, also, \cite{Bon}\,), containing the important pre-existent subclass of quaternionic--K\"ahler manifolds  
(see \cite{LeBSal-rigid_qK}\,). Furthermore, such spaces of `rational curves' (that is, nonconstant holomorphic maps from the sphere) are, also, involved 
in the classification of irreducible representations whose images can occur as holonomy groups of torsion free connections 
(\,\cite{Bry-holo_twist}\,, \cite{ChiSch-holo_twist}\,; see, also, \cite{MerSch}\,, and the references therein). Also, note that, 
the almost $\r$-quaternionic manifolds are related to the paraconformal manifolds of \cite{BaiEas-paraconf}\,.\\ 
\indent 
On the other hand, the algebraic geometers are interested in complex projective manifolds which admit an embedding of the Riemann sphere 
with positive normal bundle, as these are the `rationally connected manifolds' - see the paper \cite{Paltin-2005} of Paltin~Ionescu 
(and the references therein)  to whom I am grateful for kindly drawing my attention to this important fact.\\ 
\indent 
We hope that our approach will deepen the understanding of the interconnections of these two domains.

\section{Principal $\r$-connections} \label{section:gen_connection} 

\indent 
We work in the complex analytic category. By the tangent bundle we mean the holomorphic tangent bundle and if $E$ is a vector bundle 
then $\G(E)$ denotes the sheaf of its sections. All the facts that follow, in this section, can be straightforwardly extended to the smooth setting.\\ 
\indent 
Let $M$ be a manifold endowed with a pair $(E,\r)$\,, where $E$ is a vector bundle over $M$ and $\r:E\to TM$ is 
a morphism of vector bundles.\\ 
\indent 
Let $(P,M,G)$ be a principal bundle and let $\p:P\to M$ be the projection. A \emph{principal $\r$-connection} on $P$ 
is a $G$-equivariant vector bundle morphism $C:\p^*E\to TP$ which when composed to the morphism from $TP$ to $\p^*(TM)$\,, 
induced by $\dif\!\p$, gives $\p^*\r$.\\ 
\indent 
If $E=TM$ and $\r$ is the identity, we retrieve the classical notion of principal connection. Also, if $E$ is a distribution on $M$ 
and $\r$ is the inclusion morphism then a principal $\r$-connection is just a principal partial connection over $E$.\\ 
\indent 
Any principal $\r$-connection on $P$ corresponds to a morphism of vector bundles $c:E\to TP/G$ which when composed to $TP/G\to TM$ 
gives $\r$\,.\\ 
\indent 
Also, let $(a_{UV})_{U,V\in\mathcal{U}}$ be a cocycle defining $P$, where $\mathcal{U}$ is an open covering of $M$. 
Then any principal $\r$-connection on $P$ corresponds to a family $(A_U)_{U\in\mathcal{U}}$ of 
\emph{local connection forms}, where $A_U\in\G\bigl(\mathfrak{g}\otimes (E|_U)^*\bigr)$ and 
$A_V=\r^*\bigl(a_{UV}^*\theta\bigr)+{\rm Ad}\bigl(a_{UV}^{\,-1}\bigr)A_U$ 
on $U\cap V$\,, where $\mathfrak{g}$ is the Lie algebra of $G$, $\r^*:T^*M\to E^*$ is the transpose of $\r$\,, and $\theta$ is the canonical form on $G$. 
Over each local trivialisation $P|_U=U\times G$ we have $TP/G=(TU)\oplus\mathfrak{g}$ and $c=(\r\,,-A_U)$\,.\\    
\indent 
Obviously, any principal connection induces a principal $\r$-connection, but not all principal $\r$-connections are obtained this way. 
Also, the set of principal $\r$-connections is an affine space over the vector space of global sections of ${\rm Hom}(E,{\rm Ad}P)$\,. Consequently, 
similarly to the classical case \cite{At-57}\,, the obstruction to the existence of principal $\r$-connections is an element of $H^1\bigl(M,{\rm Hom}(E,{\rm Ad}P)\bigr)$  
(which, in the smooth case, is zero).\\ 
\indent 
Let $F$ be a bundle associated to $P$. Then any principal $\r$-connection on $P$ induces a morphism of vector bundles 
$c_F:\p^*E\to TF$ which we call a \emph{$\r$-connection} (here, $\p$ is the projection of $F$). 
If $F$ is a vector bundle then $c_F$ corresponds to a \emph{covariant derivation} $\nabla:\G(F)\to\G\bigl({\rm Hom}(E,F)\bigr)$ 
which is a linear map such that $\nabla(fs)=\r^*(\dif\!f)\otimes s+f\nabla s$\,, for any (local) function $f$ on $M$ and 
any (local) section $s\in\G(F)$. Note that, in \cite{Car-96} was considered a notion which is, essentially, the same with 
the covariant derivation of a $\r$-connection on a vector bundle.\\  
\indent 
The usual constructions of connections - such as, the pull-back and the tensor product - admit straight generalizations to $\r$-connections.\\ 
\indent 
Any $\r$-connection $\nabla$ on $E$ has a \emph{torsion} $T\in\G(TM\otimes\Lambda^2E^*)$ (compare \cite{Pan-integrab_co-cr_q}\,), characterised by 
$T(s_1,s_2)=\r\circ(\nabla_{s_1}s_2-\nabla_{s_2}s_1)-[\r\circ s_1,\r\circ s_2]$\,, for any $s_1,s_2\in\G(E)$\,. 
Note that, if $T=0$ then, on defining $[s_1,s_2]=\nabla_{s_1}s_2-\nabla_{s_2}s_1$\,, for any $s_1,s_2\in\G(E)$\,, 
we obtain that:\\ 
\indent 
\quad(i) $[\cdot,\cdot]$ is linear and skew-symmetric on $\G(E)$\,,\\ 
\indent 
\quad(ii) $\r$ intertwine $[\cdot,\cdot]$ and the usual bracket on vector fields on $M$\,,\\ 
\indent 
\quad(iii) $[s_1,fs_2]=(\r\circ s_1)(f)s_2+f[s_1,s_2]$\,, for any function $f$ on $M$ and $s_1,s_2\in\G(E)$\,.\\ 
We call $[\cdot,\cdot]$ the \emph{bracket associated to $\nabla$}.\\  
\indent 
For applications there is a characterisation of the torsion which in the classical case is a consequence of the Cartan's first structural equation. 
If $(P,M,G)$ is the frame bundle of $E$ then $\p^*E=P\times E_0$\,, where $E_0$ is the typical fibre of $E$.  
Then any principal $\r$-connection on $P$ is a $G$-invariant morphism of vector bundles $C:P\times E_0\to TP$. For any $\xi\in E_0$ let 
$B(\xi)$ be the vector field tangent to $P$ such that $B(\xi)_u=C(u,\xi)$\,, for any $u\in P$ (in the classical case, these become the well known 
`standard horizontal vector fields'; see, also, \cite{Pan-integrab_co-cr_q}\,). 

\begin{prop} \label{prop:first_Cartan} 
If $T$ is the torsion of $C$ then $T(u\xi,u\e)=-\dif\!\p\bigl([B(\xi),B(\e)]_u\bigr)$\,, for any $\xi,\e\in E_0$ and $u\in P$. 
\end{prop} 
\begin{proof} 
We may assume $P=M\times G$ and verify the relation at $u=(x,e)$\,, where $e$ is the identity of $G$. As $C$ is $G$-invariant, we have 
$B(\xi)_{(x,a)}=\bigl(\r(x,a\xi),-\G(x,a\xi)a\bigr)$\,, for any $x\in M$ and $a\in G$, where $\G$ is the corresponding local connection form.\\ 
\indent 
If $X$ and $A$ are vector fields on $P$ such that, at each point, $X$ is tangent to $M$ and $A$ is tangent to $G$ 
then, for any $x\in M$ and $a\in G$, we have that $\dif\!\p\bigl([X,A]_{(x,a)}\bigr)$ depends only of $X$ and $A_a$\,.\\ 
\indent  
Consequently, the bracket of $B(\xi)$ and the component of $B(\e)$ tangent to $G$ is mapped by $\dif\!\p_{(x,e)}$ to $\r(x,\G(x,\e)\xi)$\,. 
Hence, $$\dif\!\p\bigl([B(\xi),B(\e)]_{(x,e)}\bigr)=\r(x,\G(x,\e)\xi)-\r(x,\G(x,\xi)\e)+[\r(\cdot,\xi),\r(\cdot,\e)]_x\;,$$ 
which is easily seen to be equal to $-T\bigl((x,\xi),(x,\e)\bigr)$\,. 
\end{proof}

\indent 
Returning, now, to a principal $\r$-connection $c:E\to TP/G$ on $P$, suppose that $E$ is endowed with a bracket $[\cdot,\cdot]$\,, 
satisfying (i)\,, (ii)\,, (iii)\,, above. Then we can define the \emph{curvature form} 
$R\in\G\bigl({\rm Ad}P\otimes\Lambda^2E^*\bigr)$ of $c$\,, which, for any $s_1,s_2\in\G(E)$\,, 
is given by $R(s_1,s_2)=c\circ[s_1,s_2]-[c\circ s_1,c\circ s_2]$\,, where the second bracket is induced by the usual bracket 
on vector fields on $P$, and we have identified ${\rm Ad}P$ with the kernel of the morphism $TP/G\to TM$.\\ 
\indent  
All the classical formulae involving $R$ admit straight generalizations to this setting, up to 
the Bianchi identities for which one needs that the bracket on $E$ satisfies the Jacobi identity  
(see \cite{KMa-BLMS95} for the corresponding theory of such brackets). Then, by using \cite[Theorem 2.2]{KMa-BLMS95}\,, 
if $\r$ is surjective, at least locally, we have $E=TQ/H$ and the following commutative diagram, 
where $(Q,M,H)$ is a principal bundle: 
\begin{displaymath}
\xymatrix{
         &                       &                       &   0  \ar[d]               &  \\
         &                       &                       & {\rm Ad}Q \ar[d]                         &  \\
         &                       &                       & TQ/H  \ar[d]^{\r} \ar[dl]_{c} &  \\
0 \ar[r] & {\rm Ad}P \ar[r]^{\iota}      &   TP/G \ar[r]    &TM \ar[r] \ar[d] & 0 \\
         &                       &                               &   0                       &   }
\end{displaymath}
\indent 
There are two important particular cases, the first of which showing that the Cartan connections are just a special kind of principal $\r$-connections.  

\begin{exm} \label{exm:princ_Cartan} 
Let $H\subseteq G$ be a closed subgroup, and $Q$ a restriction of $P$ to $H$.\\ 
\indent  
Then the simplest example of a principal $\r$-connection is the canonical inclusion map 
$c_0:TQ/H\to TP/G$ (this is, obviously, flat; that is, its curvature form is zero).\\ 
\indent 
Any other principal $\r$-connection $c$ on $P$ differ from $c_0$ by a vector bundles morphism $\g:TQ/H\to{\rm Ad}P$\,; that is, 
$c=c_0-\iota\circ\g$\,.\\  
\indent 
Note that, in this setting, the Cartan connections appear as principal $\r$-connections $c$ on $P$ satisfying the following:\\  
\indent 
\quad(1) $c$ is induced by a (classical) principal connection on $P$ (that is, $c$ factors into $\r$ followed by a section of $TP/G\to TM$),\\ 
\indent 
\quad(2) $c$ restricted to ${\rm Ad}Q$ takes values in ${\rm Ad}P$ and is given by the inclusion $\mathfrak{h}\to\mathfrak{g}$\,,\\ 
\indent 
\quad(3) the corresponding $\g:TQ/H\to{\rm Ad}P$ is an isomorphism.\\ 
\indent  
It is fairly well known that for flat Cartan connections we, locally, have $Q=G$, $M=G/H$ 
and $\g$ is the canonical isomorphism of vector bundles from $TG/H$ onto $(G/H)\times\mathfrak{g}$\,. 
Then Proposition \ref{prop:first_Cartan} gives that the torsion is given 
by $T_0\in(\mathfrak{g}/\mathfrak{h})\otimes\Lambda^2\mathfrak{g}^*$ defined by $T_0(A,B)$ is the projection 
of $-[A,B]$ onto $\mathfrak{g}/\mathfrak{h}$\,. 
\end{exm} 

\indent 
The second class of examples of principal $\r$-connections is, in a certain sense, dual to Example \ref{exm:princ_Cartan}\,. 

\begin{exm} \label{exm:for_Veronese} 
Let $G\subseteq H$ be a closed subgroup, $P$ a restriction of $Q$ to $G$, and $c$ a retraction of the inclusion $TP/G\to TQ/H$. 
Then $c$ is determined by its kernel which is a vector subbundle of ${\rm Ad}Q$\,.\\  
\indent 
We shall need the case $P=K$, where $K$ is a Lie group, with $G\subseteq H\subseteq K$ closed subgroups. Then $M=K/G$ 
and we may identify $Q=(K\times H)/G$, where $G$ is embedded diagonally as a closed Lie subgroup of $K\times H$. 
Accordingly, the action of $H$ on $Q$ is obtained from the action to the left of $H$, seen as a normal subgroup of $K\times H$.\\ 
\indent 
Then $TQ/H$ is the bundle associated to $(K,K/G,G)$ through the representation 
of $G$ on $(\mathfrak{k}\times\mathfrak{h})/\mathfrak{g}$ induced by the adjoint representations of $K$ and $H$, 
where $\mathfrak{g}$\,, $\mathfrak{h}$ and $\mathfrak{k}$ are the Lie algebras of $G$, $H$ and $K$, respectively.  
Hence, $\r$ is given by the projection from $(\mathfrak{k}\times\mathfrak{h})/\mathfrak{g}$ 
onto $\mathfrak{k}/\mathfrak{g}$\,; in particular, ${\rm ker}\r$ is the subbundle corresponding 
to $(\mathfrak{g}\times\mathfrak{h})/\mathfrak{g}=\mathfrak{h}$\,.\\ 
\indent 
Thus, any $K$-invariant principal $\r$-connection $c:TQ/H\to TP/G$ which is the identity on $TP/G$ 
corresponds to a $G$\,-invariant projection $\g:(\mathfrak{k}\times\mathfrak{h})/\mathfrak{g}\to\mathfrak{k}$ 
which is the identity when restricted to $\mathfrak{k}$ 
and when composed with the projection $\mathfrak{k}\to\mathfrak{k}/\mathfrak{g}$ gives the projection   
from $(\mathfrak{k}\times\mathfrak{h})/\mathfrak{g}$ onto $\mathfrak{k}/\mathfrak{g}$\,. Therefore any such principal $\r$-connection 
corresponds to a $G$\,-invariant decomposition $\mathfrak{h}=\mathfrak{g}\oplus\mathfrak{m}$\,, for some vector subspace 
$\mathfrak{m}\subseteq\mathfrak{h}$\,, so that $c$ is induced by the projection $\g$ from 
$\mathfrak{k}\times\mathfrak{m}\bigl(=(\mathfrak{k}\times\mathfrak{h})/\mathfrak{g}\bigr)$ onto $\mathfrak{k}$ with kernel $\mathfrak{m}$\,.\\ 
\indent 
The curvature form of the principal $\r$-connection corresponding to $\mathfrak{m}$ is given by the $\mathfrak{g}$-valued 
two-form on $\mathfrak{k}\times\mathfrak{m}$ induced by the $\mathfrak{g}$-valued two-form on $\mathfrak{m}$ which is the composition 
of the restriction of the bracket on $\mathfrak{h}$ followed by the opposite of the projection on $\mathfrak{g}$\,. Furthermore, 
by using Proposition \ref{prop:first_Cartan}\,, we obtain that the torsion is given by the composition of the bracket on $\mathfrak{k}$ 
followed by the opposite of the projection on $\mathfrak{k}/\mathfrak{g}$\,.\\ 
\indent 
On comparing with Example \ref{exm:princ_Cartan}\,, we see that these $K$-invariant principal $\r$-con\-nec\-tions on $(K,K/G,G)$ 
are `extentions' of $H$-invariant (classical) principal connections  
on $(H,H/G,G)$ by the flat Cartan connection on $(K,K/G,G)$. 
\end{exm}

\section{The quaternionic geometry of the embeddings of the Riemann sphere} \label{section:r-q} 

\indent 
A \emph{(complex) quaternionic vector bundle} is a (complex analytic) vector bundle $E$ with typical fibre $\C^{\!2}\otimes \C^{\!k}$ 
and structural group ${\rm SL}(2,\C)\cdot{\rm GL}(k,\C)$ acting on the typical fibre through the tensor product of the canonical representations 
of the factors. Thus, at least locally, any quaternionic vector bundle is of the form $S\otimes F$, where $S$ and $F$ are vector bundles 
and the structural group of $S$ is ${\rm SL}(2,\C)$\,; obviously, $Z=PS$ is always globally defined.\\ 
\indent 
An \emph{almost $\r$-quaternionic structure} on a manifold $M$ is a pair $(E,\r)$\,, where $E$ is a quaternionic vector bundle over $M$,  
and $\r:E\to TM$ is a morphism of vector bundles whose kernel, at each point, contains no (nonzero) decomposable elements. 
Obviously, if $\r$ is an isomorphism then we obtain the (complexified version of the) notion of almost quaternionic structure (see \cite{Sal-dg_qm}\,), 
whilst if $\r$ is surjective then we obtain the notion of almost co-CR quaternionic structure \cite{fq_2}\,.\\ 
\indent 
Suppose that $E\,(=S\otimes F)$ is endowed with a $\r$-connection $\nabla$ (compatible with its structural group) 
and let $c_Z:\p^*E\to TZ$ be the induced $\r$-connection on $Z$, where $\p:Z\to M$ is 
the projection. Let $\mathcal{B}$ be the distribution on $Z\,(=PS)$ such that $\mathcal{B}_{[e]}=c_Z\bigl(\{e\otimes f\,|\,f\in E_{\p(e)}\}\bigr)$\,, 
for any $[e]\in Z$\,. As $c_Z$ is a $\r$-connection and the kernel of $\r$ contains no decomposable elements, $\mathcal{B}$ is well-defined; 
furthermore, $\mathcal{B}\cap({\rm ker}\dif\!\p)=0$\,.\\ 
\indent 
We say that $(M,E,\r)$ is a \emph{$\r$-quaternionic manifold} if the $\r$-connection $\nabla$ can be chosen such that $\mathcal{B}$ is integrable; 
then $\nabla$ is a \emph{quaternionic $\r$-connection}. If, further, there exists a surjective submersion 
$\p_Y:Z\to Y$ such that ${\rm ker}\dif\!\p_Y=\mathcal{B}$ and each fibre of $\p_Y$ intersects at most once each fibre of $\p$, then $Y$ endowed 
with $\bigl\{\p_Y\bigl(\p^{-1}(x)\bigr)\bigr\}_{x\in M}$\,, is the \emph{twistor space} of  $(M,E,\r,\nabla)$\,.
Obviously, if $\r$ is surjective and $\nabla$ is a (classical) connection then we retrieve the co-CR quaternionic manifolds \cite{fq_2}\,. 

\begin{thm} \label{thm:r-q} 
There exists a natural bijective correspondence between the following, where $Y$ is a complex manifold:\\ 
\indent 
{\rm (i)} germs of (complex analytic) embeddings of the Riemann sphere into $Y$ with nonnegative normal bundles;\\ 
\indent 
{\rm (ii)} germs of $\r$-quaternionic manifolds whose twistor spaces are open subsets of $Y$. 
\end{thm}  
\begin{proof} 
Let $Y$ be a manifold endowed with an embedded Riemann sphere $t_0\subseteq Y$, with nonnegative normal bundle. 
Then, by \cite{Kod}\,, there exist a map $\p_Y:Z\to Y$ and a surjective proper submersion $\p:Z\to M$ such that 
$\p_Y$ restricted to each fibre of $\p$ is an embedding, and 
$t_0=\p_Y\bigl(\p^{-1}(x_0)\bigr)$\,, for some $x_0\in M$; moreover,  
$\p_Y$ and $\p$ induce linear isomorphisms $T_xM=H^0(t_x,Nt_x)$\,, where $t_x=\p_Y\bigl(\p^{-1}(x)\bigr)$\,,   
and $Nt_x$ is its normal bundle, $(x\in M)$\,. Furthermore, by the rigidity of the Riemann sphere (see \cite{Nar-deformations}\,), 
$\p$ is locally trivial. \\ 
\indent 
As $Nt_0$ is nonnegative, it is generated by its sections. Consequently, $\dif\!\p_Y$ restricted to $\p^{-1}(x_0)$ is submersive. 
Thus, by passing, if necessary, to open neighbourhoods of $x_0$ and $t_0$ we may assume $\p_Y$ a surjective submersion. 
Therefore, for any $x\in M$, we have an exact sequence 
\begin{equation} \label{e:exact_sequence} 
0\longrightarrow\mathcal{B}_x\longrightarrow t_x\times T_xM\longrightarrow Nt_x\longrightarrow0\;, 
\end{equation}  
where we have identified $t_x=\p^{-1}(x)$\,, we have denoted $\mathcal{B}={\rm ker}\dif\!\p_Y$, $\mathcal{B}_x=\mathcal{B}|_{t_x}$\,, 
and the inclusion morphism $\mathcal{B}_x\to T_xM$ is induced by $\dif\!\p$\,. Together with the fact that the induced linear map $T_xM\to H^0(t_x,Nt_x)$ 
is bijective, this shows that $Nt_x$ is nonnegative, for any $x\in M$. Furthermore, from the exact sequence of cohomology groups, 
induced by \eqref{e:exact_sequence}\,, we deduce that both $H^0(t_x,\mathcal{B}_x)$ and $H^1(t_x,\mathcal{B}_x)$ are zero. 
Therefore the Birkhoff--Grothendieck decomposition of $\mathcal{B}_x$ contains only terms of Chern number~$-1$\,.\\ 
\indent 
Consequently, by \cite[Theorems 7.4]{KodSpe-I_II} and \cite[Theorem 9]{KodSpe-III} (see, also, \cite{GraRem-cas}\,,\,\cite{Nar-deformations}\,), 
the dual of the direct image, through $\p$, of $\mathcal{B}^*$ is a vector bundle over $M$ which, by also using \cite[\S3]{fq}\,, it is 
a quaternionic vector bundle. Denote this quaternionic vector bundle by $E$ and, note that, $\p^*E$ contains (the annihilator of) $\mathcal{B}$ 
so that we have an exact sequence $0\longrightarrow\mathcal{B}\longrightarrow\p^*(E)\longrightarrow\mathcal{E}\longrightarrow0$\,, 
for some vector bundle $\mathcal{E}$ over $Z$.\\ 
\indent 
The exact sequence of cohomology groups of the dual of \eqref{e:exact_sequence} induces a linear map $\r^*_x:T_x^*M\to H^0(t_x,\mathcal{B}_x^*)=E_x^*$\,, 
for any $x\in M$. By applying, for example, \cite[\S3]{fq}\,, we obtain that $(E,\r)$ is an almost $\r$-quaternionic structure on $M$. 
Note that, $\r$ is induced by a morphism of vector bundles $R:\mathcal{E}\to\p^*(TM)/\mathcal{B}$.\\  
\indent 
At least locally, we may suppose $E=S\otimes F$ with $S$ of rank $2$ and whose structural group is ${\rm SL}(2,\C)$\,. 
Also, $Z=PS$ and the fibre, over $x$\,, of the quotient of the tangent bundle of ${\rm SL}(S)$\,, through ${\rm SL}(2,\C)$\,,  
is $H^0\bigl(t_x,(TZ)|_{t_x}\bigr)$\,. Therefore, to prove the existence of a quaternionic $\r$-connection on $E$ which determines $Y$, 
it is sufficient to find a morphism $\p^*E\to TZ$ which when restricted to $\mathcal{B}$ is the identity, and when composed 
with the morphism $TZ\to\p^*(TM)$\,, induced by $\dif\!\p$, it is equal to $\p^*\r$\,; equivalently, we have to find a morphism 
$\mathcal{E}\to TZ/\mathcal{B}$ which when composed to $TZ/\mathcal{B}\to\p^*(TM)/\mathcal{B}$ 
is equal to $R$\,.\\ 
\indent 
Now, the obstruction to build such a morphism is an element of $H^1(Z,\mathcal{E}^*\otimes{\rm ker}\dif\!\p)$\,. As $Z$ is locally trivial, 
and the restriction of $\mathcal{E}^*\otimes{\rm ker}\dif\!\p$ to each fibre of $\p$ is isomorphic to $k\ol(1)$\,, where $\rank\mathcal{E}=k$\,, 
an application of the K\"unneth formula shows that any point of $M$ has an open neighbourhood $U$ such that the restriction of this 
obstruction to $Z|_U$ is zero.\\ 
\indent 
Conversely, suppose that an open subset of $Y$ is the twistor space of a $\r$-quaternionic manifold $M$. Then, for all $x\in M$, we obtain exact sequences 
as \eqref{e:exact_sequence}\,, where now we know that the Birkhoff--Grothendieck decomposition 
of $\mathcal{B}_x$ contains only terms of Chern number $-1$\,; hence, $T_xM=H^0(t_x,Nt_x)$ and, consequently, $Nt_x$ is nonnegative. 
\end{proof} 

\indent 
Note that, Theorem \ref{thm:r-q} gives, in fact, an isomorphism of categories, as the families of embedded Riemann spheres involved are locally complete. 
In particular, any $\r$-quaternionic vector space (that is, a $\r$-quaternionic vector bundle over a point) corresponds to a nonnegative vector bundle 
over the sphere (cf.\ \cite{Qui-QJM98}\,; see, also, \cite{Pan-vq}\,).\\ 
\indent 
If we apply Theorem \ref{thm:r-q} to the Veronese curve in the complex projective space, of dimension at least two, then we obtain a $\r$-quaternionic manifold 
$(M,E,\r)$ endowed with a quaternionic $\r$-connection which is not a (classical) connection. Indeed, if that would have been the case then the normal bundle sequence 
of the Veronese curve would have split, which would contradict \cite{MorrRos-MathAnn78}\,. In fact, our approach gives, in particular, a new proof for that 
consequence of \cite{MorrRos-MathAnn78}\,, as we shall now explain. 

\begin{exm} \label{exm:qgfs-Veronese} 
Let $U_k=\odot^kU_1$\,, where $\dim U_1=2$\,, and $\odot$ is the symmetric product, $(k\in\mathbb{N})$\,. Then $U_k$ 
is the irreducible representation of dimension $k+1$ of ${\rm SL}(U_1)$\,.\\ 
\indent  
For any $k\geq1$\,, we have an embedding ${\rm PGL}(U_1)\subseteq{\rm PGL}(U_k)$ and the quotient ${\rm PGL}(U_k)/{\rm PGL}(U_1)$\,, 
obviously, parametrizes the space of Veronese curves on $PU_k$\,.\\ 
\indent 
We are in the setting of Example \ref{exm:for_Veronese} with $G={\rm PGL}(U_1)$\,, $H={\rm PGL}(U_{k-1})$\,, and $K={\rm PGL}(U_k)$\,, 
where $k\geq2$ (the case $k=1$ is trivial). To describe the quaternionic $\r$-connection, note that, the representation of $G$ on the Lie algebra $\mathfrak{k}$ 
of $K$,  induced by the adjoint representation of $K$, is $U_2\oplus\cdots\oplus U_{2k}$\,; indeed, as $\mathfrak{k}\oplus U_0=U_k\otimes U_k$\,, 
we may apply \cite[p.\ 151]{FulHar}\,. Accordingly, for the Lie algebra $\mathfrak{h}$ of $H$ we have 
$\mathfrak{h}=U_2\oplus\cdots\oplus U_{2k-2}$\,, and the quaternionic $\r$-connection is given by $\mathfrak{m}=U_4\oplus\cdots\oplus U_{2k-2}$\,. 
(Note that, if $k=2$ then we, also, are in the setting of Example \ref{exm:princ_Cartan} with $G={\rm PGL}(U_2)$\,, $H={\rm PGL}(U_1)$ 
and the quaternionic $\r$-connection given by the flat Cartan connection on $(G,G/H,H)$\,.)\\ 
\indent 
From the proof of Theorem \ref{thm:r-q} it follows that the invariant principal $\r$-connection on $(K,K/G,G)$ is unique. 
As this is not a classical principal connection, the normal bundle sequence of a Veronese curve does not split. 
\end{exm} 

\indent 
Finally, let $Y$ be the twistor space of a $\r$-quaternionic manifold $(M,E,\r)$\,, with $\r$ surjective, $\rank E>4$\,. Suppose that 
$E$ is endowed with a bracket which satisfies the Jacobi identity. By, also, applying 
\cite[Theorem 2.2]{KMa-BLMS95}\,, we obtain that, at least locally, there exists a quaternionic manifold $Q$ and a 
twistorial submersion $\phi:Q\to M$ (that is, $\phi$ corresponds to a submersion from the twistor space of $Q$ onto $Y$).  
For example, if $M$ is the space of Veronese curves of degree $k$ then $Q=\bigl({\rm PGL}(U_k)\times{\rm PGL}(U_{k-1})\bigr)/{\rm PGL}(U_1)$\,. 
By, also, taking into consideration the conjugations we obtain  real quaternionic structures on ${\rm SU}(3)$ (see \cite{Joyce-quatern_consrt}\,), 
and $\bigl({\rm SU}(k+1)\times{\rm SU}(k)\bigr)/{\rm SU}(2)$\,.

\end{document}